\documentclass[12pt]{amsart}

\usepackage[utf8]{inputenc}

\usepackage{tikz-cd}

\usepackage{amsmath}
\usepackage{amssymb}
\usepackage{amscd}
\usepackage{amsthm}
\usepackage{textcomp}
\usepackage{xcolor}
\usepackage{cancel}

\usepackage{accents}

\usepackage{latexsym}
\usepackage{mathtools}
\usepackage{stmaryrd}

\def\tr{\operatorname{Tr}}

\def\id{\operatorname{id}}

\def\Rz{\mathbb R}
\def\R{\mathbb R}
\def\Rz{\mathbb R}

\def\Cc{\mathcal C}

\def\Lc{\mathcal L}

\def\Xf{\mathfrak X}

\def\R{{\mathbb R}}

\def\0{{\mathbf 0}}
\def\1{{\mathbf 1}}

\def\Ric{{\operatorname{Ric}}}

\def\Sym{{\operatorname{Sym}}}

\def\GL{{\operatorname{GL}}}

\def\de{\delta}

\def\th{\theta}

\def\ps{\psi}

\def\Ga{\Gamma}
\def\De{\Delta}

\def\ad{\operatorname{ad}}

\def\aequiv#1{\;\raise-4pt\hbox{$\substack{{\displaystyle\sim}\\\raise4pt\hbox{$\scriptstyle #1$}}$}
\;}

\DeclareMathAccent{\wtilde}{\mathord}{largesymbols}{"65}

\def\ric#1{\operatorname{Ric}^{#1}}
\makeatletter
\def\cn#1{\def\tmp@cn{#1}\ifx\tmp@cn\empty\nabla\else\nabla^{\tmp@cn}\fi}
\makeatother

\def\scal#1{\operatorname{Scal}^{#1}}
\def\lie{\mathcal{L}}

\def\gf{\mathfrak g}

\newtheorem{theorem}{Theorem}[section]
\newtheorem{lem}[theorem]{Lemma}
\newtheorem{pr}[theorem]{Proposition}
\newtheorem{co}[theorem]{Corollary}
\theoremstyle{definition}
\newtheorem{defi}[theorem]{Definition}
\newtheorem{ex}[theorem]{Example}
\newtheorem{re}[theorem]{Remark}
\newcommand{\mg}{\mathfrak g }

\newcommand{\mm}{\mathfrak m }

\newcommand{\mz}{\mathfrak z }

\newcommand{\mh}{\mathfrak h }

\newcommand{\lela}{ g(}
\newcommand{\rira}{)}

\renewcommand{\R}{\mathbb R}

\newcommand{\nc}{\newcommand}
\nc{\Iso}{\operatorname{Iso}}
 \nc{\iso}{\mathfrak{iso}}
 \nc{\sso}{\mathfrak{so}}
\nc{\Ad}{\operatorname{Ad}} 
 \nc{\Dera}{\operatorname{Dera}} \nc{\Auto}{\operatorname{Auto}}
 \nc{\LL}{{\rm L}}
\nc{\dd}{{\rm d}}
\nc{\Id}{{\rm Id}}

\begin{document}

\title{Weyl-Einstein structures on conformal  solvmanifolds}

\author{Viviana del Barco}
\address{V.~del Barco: Instituto de Matemática, Estatística e Computação Científica, Universidade Estadual de Campinas,  Rua Sergio Buarque de Holanda, 651, Cidade Universitaria Zeferino Vaz, 13083-859, Campinas, São Paulo, Brazil.}

\thanks{V.~del Barco is partially supported by  FAEPEX-UNICAMP grant 2566/21 and FAPESP grant 2021/09197-8; V.~del Barco and A.~Moroianu are supported by MATHAMSUD Regional Program 21-MATH-06.}

\email{delbarc@unicamp.br}

\author{Andrei Moroianu}
\address{A.~Moroianu: Université Paris-Saclay, CNRS,  Laboratoire de mathématiques d'Orsay, 91405, Orsay, France}
\email{andrei.moroianu@math.cnrs.fr}

\author{Arthur Schichl}
\address{A.~Schichl: École polytechnique, 91120, Palaiseau, France, and ETH Zürich, 8006, Zürich, Switzerland}
\email{aschichl@student.ethz.ch}

\date{\today}

\begin{abstract} A conformal Lie group is a conformal manifold $(M,c)$ such that $M$ has a Lie group structure and $c$ is the conformal structure defined by a left-invariant metric $g$ on $M$. 
We study Weyl-Einstein structures on conformal solvable Lie groups and on their compact quotients. In the compact case, we show that every conformal solvmanifold carrying a Weyl-Einstein structure is Einstein. We also show that there are no left-invariant Weyl-Einstein structures on non-abelian nilpotent conformal Lie groups, and classify them on conformal solvable Lie groups in the almost abelian case. Furthermore, we determine the precise list (up to automorphisms) of left-invariant metrics on simply connected solvable Lie groups of dimension 3 carrying left-invariant Weyl-Einstein structures.
\end{abstract}


\maketitle

\section{Introduction}

In $1964$, Wolf~\cite{Wo64} showed that there are no left-invariant Einstein metrics on non-abelian nilpotent Lie groups (see also \cite{Mil76}). Our aim is to study a similar question in the conformal setting. 

It was Hermann Weyl \cite{We23,Tr16} who first introduced a generalization of Einstein metrics in conformal geometry. A Weyl-Einstein structure on a conformal manifold $(M,c)$ is a torsion-free linear connexion preserving the conformal structure, whose symmetric trace-free component of the Ricci tensor vanishes. Unlike the Riemannian case, it is still unknown whether there exists a non-abelian conformal nilpotent Lie group $(M,c)$, where $c$ contains a left-invariant metric, carrying a Weyl-Einstein structure. 

The notion of nilmanifolds as manifolds endowed with a transitive action of a nilpotent Lie group was introduced by Mal'cev~\cite{Mal49} in $1949$. Mal'cev has proved subsequently that every compact nilmanifold $M$ can be obtained by taking the quotient of a nilpotent Lie group $G$ by some discrete subgroup $\Gamma\subset G$, that is, $M=\Gamma\backslash G$. Every left-invariant Riemannian metric $g$ on $G$ defines a unique Riemannian metric $\bar g$ on $M$ such that the projection $G\to M$ is a local isometry. 
A compact nilmanifold $M$ endowed with a conformal structure $c$ that contains such a metric $\bar g$ is called a compact conformal nilmanifold. 

More generally, if a compact manifold $M$ is the quotient of a simply connected solvable Lie group $G$ by a discrete co-compact subgroup $\Gamma\subset G$, and $c$ is the conformal structure defined by the projection to $M$ of a left-invariant Riemannian metric on $G$, we say that $(M,c)$ is a compact conformal solvmanifold. 

Our first result (Theorem \ref{th:WEiffE} and Corollary \ref{cor:42}) is that a compact conformal solvmanifold $(M,c)$ does not admit any Weyl-Einstein structure, unless its universal cover  endowed with the corresponding left-invariant metric $g$ is flat. The detailed proof can be found in Section~\ref{ch:solv}, which is preceded by the preliminary Section~\ref{ch:prelim} listing some general results in conformal geometry and by Section~\ref{s:conform} where we give a quick introduction to conformal geometry on Riemannian Lie groups, with special emphasis on the notion of Weyl--Einstein structures. 

In the remaining part of the article, we drop the compactness assumption, but impose the left-invariance of the Weyl structure instead. More precisely, we study Weyl-Einstein structures on Riemannian Lie groups $(G,g)$ whose Lee form with respect to $g$ is left-invariant. This problem can be stated in purely algebraic terms on the Lie algebra of $G$, but is intractable in full generality. When $G$ is nilpotent, we show that no such structure exists, unless $G$ is abelian (Proposition \ref{pro:nilWE}). The situation is more complicated in the solvable case, where even left-invariant Einstein metrics do actually exist. 

In Section \ref{ch:aa} we obtain the classification of left-invariant Weyl-Einstein structures on almost abelian conformal Lie groups  (Theorem \ref{pro:aa}) and in Section \ref{ch:3d} we study the 3-dimensional solvable case  (Theorem \ref{th:cl3}).

\section{Preliminaries on Riemannian and conformal geometry}\label{ch:prelim}

This preliminary section aims to set notations and provide the basic facts that will be of use for our main results. We include a brief introduction to conformal geometry and classical formulas in Riemannian and conformal geometry. For further information on these topics, we refer the reader to  Besse~\cite{Besse2008}, Bourguignon et al.~\cite{BHMMM}, Gauduchon~\cite{Ga95} and Moroianu~\cite{Mor07}.

Let $(M,g)$ be a Riemannian manifold of dimension $n$. We assume throughout the paper that $n\geq 3$. Let $\nabla^g$  denote  the Levi-Civita connection of $g$, whose curvature tensor  $R^g$  is defined by 
\begin{equation}\label{eq:riemc}
R^g(X,Y)Z:=\nabla^g_{[X,Y]}Z-[\nabla^g_X,\nabla^g_Y]Z,\qquad \mbox{ for any }
X,Y,Z \in \Xf(M).
\end{equation}
The Riemannian Ricci tensor $\ric{\nabla^g}$ is the symmetric tensor given by
\begin{equation}
\label{eq:riccg}
\ric{\nabla^g}(X,Y)={\rm Tr}(Z\mapsto R^g(X,Z)Y),\qquad \mbox{ for any }
X,Y \in \Xf(M).
\end{equation}The scalar curvature of $(M,g)$ is the trace of the Ricci tensor, that is, $\scal{g}={\rm Tr}(\ric{\nabla^g})$.

\begin{defi}\label{defi:E}
The Riemannian manifold $(M,g)$ is called Einstein if $\ric{\nabla^g}=\frac1n \scal{g} g$. Equivalently, $(M,g)$ is Einstein if the trace-free part of $\ric{\nabla^g}$ vanishes. 
\end{defi}
It is well known that any Einstein manifold $(M,g)$ of dimension $n\ge 3$ has constant scalar curvature \cite{Besse2008}.

Given a local $g$-orthonormal frame  $\{ e_i\}_{i=1}^n$ of $M$,  the co-differential of a symmetric $k$-tensor field  $\alpha$ on $M$ is the symmetric 
$(k-1)$-tensor defined as
\begin{equation}
\label{eq:codif}
\delta^g\alpha=-\sum_{i=1}^n e_i\lrcorner \nabla^g_{e_i}\alpha.
\end{equation}
The well known contracted Bianchi identity \cite[\S 1.95]{Besse2008} reads:
\begin{equation}\label{eq:bianccontr}
    %\tag{BC}
   {\rm d} \scal{g} = -2\de^g\ric{\cn{g}}.
\end{equation}

Let $\De^g$ denote the Hodge--Laplace operator. On $1$-forms, the Laplacian $\De^g$ is related to the Ricci tensor $\ric{\nabla^g}$ via the Bochner--Weitzenb\"ock formula \cite[\S 3, Ch. 7]{Pet06}:
\begin{equation}\label{eq:WF}
\De^g\th =(\cn{g})^*\cn{g}\th + \ric{\cn{g}}(\th),\qquad\mbox{ for any } \theta\in\Omega^1(M),
\end{equation}
where,
in a local $g$-orthonormal frame $\{e_i\}_{i=1}^n$ of $M$, 
\[
(\cn{g})^*\cn{g}\th : = -\sum_{i=1}^n \cn{g}_{e_i} \cn{g}_{e_i} \th.
\]

For any $k \in \Rz$, we denote $\Lc^k$ the weight bundle of weight $k$ over $M$. Recall that $\Lc^k \coloneqq \GL(M)\times_{|\det|^{\frac{k}{n}}}\Rz$ is an oriented (thus trivializable) real line bundle and that for all $k,l \in \Rz$, $\Lc^k \otimes \Lc^l \simeq \Lc^{k+l}$. Denote the bundle of all symmetric $2$-tensors over $M$ by $\Sym^2(T^*M)$ and for all $k\in\Rz$, the set of all positive elements of $\Lc^k$ by $\Lc_+^k$.

\begin{defi}
A conformal class on $M$ is a section $c$ of the bundle $\Sym^2(T^*M) \otimes \Lc^2$ over $M$, which satisfies $c(X,X) \in \Lc^2_+$ for every non-vanishing vector $X \in TM$. The pair $(M,c)$ is called a conformal manifold.
\end{defi}

Given a conformal class $c$, for every section $l\in\Ga(\Lc_+)$ there exists a unique Riemannian metric $g$ on $M$ such that $c = g \otimes l^2$; all such metrics $g$ are said to belong to the conformal class $c$. On a conformal manifold $(M,c)$, the analogous of the Levi-Civita connection is the class of so-called Weyl-structures, which we shall introduce next.

\begin{defi} A {\em Weyl structure} on a conformal manifold $(M,c)$ is a torsion-free linear connection $\nabla$ on $M$ satisfying $\cn{}c = 0$.
\end{defi}

The fundamental theorem of conformal geometry \cite{We23} states that on any conformal manifold $(M,c)$ there is a one-to-one correspondence between Weyl structures $\nabla$ and connections $\cn{\Lc}$ on $\mathcal L$. This correspondence 
is implicitly given by the conformal analogue of the Koszul formula
\begin{equation}\label{eq:fundconfgeo}
\begin{aligned}
c(\cn{}_X Y, Z) &= \frac{1}{2}(\cn{\Lc}_X(c(Y,Z)) + \cn{\Lc}_Y(c(X,Z)) - \cn{\Lc}_Z(c(X,Y)) \\ &+ c(Z,[X,Y]) - c(Y,[X,Z]) - c(X,[Y,Z])),
\end{aligned}
\end{equation}
for all $X,Y,Z \in \Xf(M)$.

Let $\cn{\Lc}$ be a connection on $\Lc$ and $g$ a metric in $c$. Then there exists a unique section $l_g \in \Ga(\Lc^+)$ such that $c = g \otimes l_g^2$ and a $1$-form $\th_g \in \Ga(T^*M)$, called the Lee form of $\cn{}$ with respect to $g$, satisfying
\begin{equation}\label{eq:thdef}
\cn{\Lc}_X l_g = \th_g(X) l_g,\qquad \mbox{ for any }X \in \Xf(M).
\end{equation}

Notice that $\th_g$ depends on the choice of $g\in c$. However, any two Lee forms of a Weyl connection (that is, corresponding to two Riemannian metrics in $c$) are cohomologous. Indeed, if $\tilde g$ is another Riemannian metric  in the class $c$, then $g=e^{-2f} \tilde g$ for some $f\in C^\infty(M)$. We can thus write $c=g \otimes l_g^2=\tilde g\otimes l_{\tilde g}^2$, so $l_g=e^{f}l_{\tilde g}$ and thus \eqref{eq:thdef} implies
\begin{equation}
	\th_g = \th_{\tilde g} + {\rm d}f.
\end{equation}
In particular, the $2$-form $F\coloneqq {\rm d}\th_g={\rm d} \th_{\tilde g}$ is independent of the choice of $g\in c$, and is called the {\em Faraday form} of $\cn{}$.  The Weyl structure $\cn{}$ is called closed when $F=0$ and exact if $\th_g$ is exact for some (and thus all) $g\in c$. When it is understood from the context, and in order to avoid heavy notation, we will denote simply by $\th$ the Lee form $\th_g$ of a metric $g\in c$.

Given a Weyl structure $\nabla$ on $(M,c)$ and a Riemannian metric $g\in c$ with Lee form $\th$, it  follows directly from  \eqref{eq:thdef} that
\begin{equation}\label{eq:Dg}
    \cn{}g = -2\th\otimes g.
\end{equation}

Notice that if $\nabla$ is exact, then there exists a Riemannian metric $g$ in $c$ whose Lee form vanishes, and thus $\nabla=\nabla^g$.

Fixing a Riemannian metric $g$ in the conformal class $c$, the Weyl connection is completely determined by the Levi-Civita connection $\nabla^g$ and the Lee form $\th$ associated to $(\nabla,g)$ through the formula
\begin{equation}\label{eq:weylrel}
       \cn{}_X Y = \cn{g}_X Y + \th(X)Y + \th(Y)X - g(X,Y)T,
\end{equation} 
for all $X,Y \in \Xf(M)$, where $T\in\Xf(M)$ is the vector field dual to $\th$ with respect to $g$. 

Conversely, any differential 1-form $\th$ on $M$ together with a Riemannian metric $g\in c$ defines a Weyl structure on $(M,c)$ via \eqref{eq:weylrel}.

The curvature tensor $R^\nabla$ of $\cn{}$ is defined by 
\begin{equation}
R^\nabla(X,Y)Z: = \cn{}_{[X,Y]}Z-\cn{}_X\cn{}_YZ+\cn{}_Y\cn{}_XZ,\qquad \mbox{ for any }
X,Y,Z \in \Xf(M).
\end{equation}

The Ricci tensor $\ric{\cn{}}$ of $\cn{}$ is the bilinear form on $TM$ defined by
\begin{equation}\label{eq:riccn}
\ric{\cn{}}(X,Y):={\rm Tr}(Z\mapsto R^{\nabla}(X,Z)Y),\qquad \mbox{ for any }
X,Y \in \Xf(M).
\end{equation}
Notice that, unlike the Riemannian Ricci tensor, $\ric{\cn{}}$ is not symmetric in general.

For every metric $g$ in $c$, the scalar curvature of $\cn{}$ is the function
$\scal{\cn{}}_{g}$ defined as the $g$-trace of $\ric{\cn{}}$.

Equation \eqref{eq:weylrel} yields after some straightforward computations that the following relations link the Ricci and scalar curvatures of $\nabla$ with those of $\nabla^g$:
\begin{equation}
\label{eq:ricrel}
    %\tag{Ric}
    \ric{\cn{}} = \ric{\cn{g}} - (n-2)\bigl(\cn{g}\th - \th \otimes \th\bigr) + \bigl(\de^g\th - (n-2)\|\th\|_g^2\bigr)g,
\end{equation}
and
\begin{equation}\label{eq:scalrel}
    %\tag{Scal}
    \scal{\cn{}}_{g} = \scal{g} + 2(n-1)\de^g\th - (n-1)(n-2)\|\th\|_g^2.
\end{equation}

It is clear from \eqref{eq:ricrel} that $\ric{\cn{}} $ is, in general, not a symmetric tensor. Instead, its skew-symmetric part is given by the skew-symmetric part of $-(n-2)\nabla^g\theta$, namely $-\frac{n-2}2 {\rm d}\theta$. Therefore, $\ric{\cn{}}$ is symmetric if and only if $\theta$ is closed, i.e. $F=0$.

This suggests that in order to generalize the Einstein condition from Riemannian connections (see Definition \ref{defi:E}) to Weyl connections, one should require the trace-free {\em symmetric part} of $\ric{\cn{}}$ to vanish. Equivalently, by \eqref{eq:ricrel} and \eqref{eq:scalrel}, one has:

\begin{defi} A Weyl-structure $\cn{}$ on a conformal manifold $(M,c)$ is said to satisfy the {\em Weyl-Einstein condition} if the following  equation holds
\begin{equation}\label{eq:EW}
	%\tag{EW}
    \begin{aligned}
        \ric{\cn{}} %&= \frac{1}{n}\scal{\cn{}} c - \frac{(n-2)}{2}F  \\
        				&= \frac{1}{n}\scal{\cn{}}_g g - \frac{(n-2)}{2}F,
    \end{aligned}
\end{equation} for some (and thus every) $g\in c$.
\end{defi}

Using \eqref{eq:ricrel} and \eqref{eq:scalrel}, we can interpret the Weyl-Einstein condition in terms of any metric $g$ in the conformal class as follows. If $\theta$ is the Lee form of $\nabla$ with respect to $g$, then $\nabla$ is Weyl-Einstein if and only if 
\begin{equation}\label{eq:ricrelwe}
    %\tag{Ric}
 \ric{\cn{g}} =(n-2)\bigl((\cn{g}\th)^{\rm{sym}} - \th \otimes \th\bigr) +\frac1n(\scal{g} + (n-2)\de^g\th+(n-2)\|\th\|_g^2)g,
\end{equation}where $(\cn{g}\th)^{\rm{sym}}$ denotes the symmetric part of $\nabla^g \th$.

\begin{re}\label{rem:EvsWE}
\begin{enumerate}
\item Suppose that $(M,g)$ is an Einstein manifold and let $c$ be the conformal class of $g$. Then the Weyl structure on $(M,c)$ defined by $g$ and $\theta=0$ in \eqref{eq:weylrel} is Weyl-Einstein. 

\item Let  $\nabla$ be an exact Weyl structure on a conformal manifold $(M,c)$. Condition \eqref{eq:EW} is equivalent to requiring that the metric $g\in c$ whose Lee form vanishes is Einstein, and thus any metric in $c$ is conformally Einstein.
\end{enumerate}
\end{re}

Recall that on compact conformal manifolds, the existence of a metric $g\in c$ whose Lee form is co-closed is guaranteed whenever the dimension is $>2$:

\begin{theorem}\cite[\S I.20]{Ga84}\label{th:gaudmet}
Let $(M,c)$ be a compact conformal manifold of dimension $n > 2$ and let $\cn{}$ be a Weyl structure on $(M,c)$. Then there exists a metric $g_0$ in $c$, unique up to multiplication by a real number, such that the Lee-form $\th_0$ of $\cn{}$ associated to $g_0$ satisfies $\de^{g_0}\th_0 = 0$. 
\end{theorem}

The metric $g_0$ in the above theorem is usually called the {\em Gauduchon metric} of $\cn{}$.

We finish the section by proving three technical lemmas that will be used in the sequel. These results appeared in the work of Gauduchon \cite{Ga95}, with  misprints on some coefficients. For the sake of completeness of our work, we include here the statements with the proper coefficients and their full proofs.

\begin{lem}\cite[Lemme 3]{Ga95}\label{lem:L3}
Let $\cn{}$ be a Weyl-Einstein structure on a conformal manifold $(M,c)$ of dimension $n>2$, and let $g$ be a metric in $c$. Then the following relation holds:
\begin{equation}\label{eq:gauduchonformel}
	  \frac{1}{n}{\rm d}\scal{\cn{}}_g - 2{\rm d}\de^g\th- 2(\de^g\th)\th  + 2\de^g(\cn{g}\th) - \de^g {\rm d}\th + 2\cn{g}_T \th + (n-3){\rm d}(\|\th\|_g^2)= 0,
\end{equation}
where $T$ is the $g$-dual of $\theta$. 
\end{lem}
\begin{proof}
Let $\{e_i\}_{i=1}^n$ be a $g$-orthonormal local frame of $TM$ and  $\{e^i\}_{i=1}^n$ its dual basis of $T^*M$. For any $f\in \Cc^{\infty}(M)$,   \eqref{eq:codif} gives
\begin{equation}\label{eq:codiff1}
	\begin{aligned}
		\de^g(fg) &= - \sum_{i=1}^n e_i \lrcorner \cn{g}_{e_i}(fg) = - \sum_{i=1}^n e_i(f)(e_i \lrcorner g) - \sum_{i=1}^n f( e_i \lrcorner \cn{g}_{e_i} g)  \\
		&= - \sum_{i=1}^n e_i(f)(e_i \lrcorner g)= -\sum_{i=1}^n {\rm d}f(e_i)e_i^*= -{\rm d}f,
		\end{aligned}
\end{equation}where we used that $\cn{g}$ is the Levi-Civita connection of $g$. 
Furthermore, 
\begin{equation}\label{eq:codiff2}
	\begin{aligned}
		\de^g(\th \otimes \th) &= - \sum_{i=1}^n e_i \lrcorner \cn{g}_{e_i}(\th \otimes \th) \\ &= - \sum_{i=1}^n \cn{g}_{e_i}\th(e_i)\th - \sum_{i=1}^n \th(e_i) \cn{g}_{e_i}\th  \\
		&= (\de^g \th) \th - \cn{g}_T \th.
	\end{aligned}
\end{equation}
On the one hand, by \eqref{eq:codiff1} and \eqref{eq:codiff2}, the co-differential $\de^g$ applied to (\ref{eq:ricrel}) yields 
\begin{equation}
	\de^g\ric{\cn{}} = \de^g\ric{\cn{g}} - (n-2)\bigl(\de^g(\cn{g}\th) - (\de^g \th) \th + \cn{g}_T \th) -{\rm d}\de^g\th + (n-2){\rm d}(\|\th\|_g^2).
\end{equation}
On the other hand, (\ref{eq:EW}) and \eqref{eq:codiff1} imply
\begin{equation}
	\begin{aligned}
	\de^g\ric{\cn{}} = -\frac{1}{n} {\rm d}\scal{\cn{}}_g - \frac{(n-2)}{2}\de^g {\rm d}\th.
	\end{aligned}
\end{equation}
Therefore, 
\begin{multline}\label{eq:deltaricg}
	 \de^g\ric{\cn{g}}  = (n-2)\bigl(\de^g(\cn{g}\th) - (\de^g \th) \th + \cn{g}_T \th)  + {\rm d}\de^g\th \\ - (n-2){\rm d}(\|\th\|_g^2) - \frac{1}{n} {\rm d}\scal{\cn{}}_g - \frac{(n-2)}{2}\de^g {\rm d}\th.
\end{multline}

Moreover, applying the exterior differential to (\ref{eq:scalrel}) and using \eqref{eq:bianccontr} yields
\begin{equation}\label{eq:dscal}
	\begin{aligned}
		{\rm d}\scal{\cn{}}_g = -2\de^g \ric{\cn{g}} + 2(n-1){\rm d}\de^g \th - (n-1)(n-2){\rm d}(\|\th\|^2_g).
	\end{aligned}
\end{equation}
Isolating $\de^g \ric{\cn{g}} $ in \eqref{eq:dscal} and using \eqref{eq:deltaricg} we get
\begin{multline}
	 	 -\frac12{\rm d}\scal{\cn{}}_g +(n-1){\rm d}\de^g \th -\frac{ (n-1)(n-2)}2{\rm d}(\|\th\|^2_g)	\\
	 = (n-2)\bigl(\de^g(\cn{g}\th) - (\de^g \th) \th + \cn{g}_T \th)  + {\rm d}\de^g\th  - (n-2){\rm d}(\|\th\|_g^2)\\ - \frac{1}{n} {\rm d}\scal{\cn{}}_g - \frac{(n-2)}{2}\de^g {\rm d}\th.
	 \end{multline}
Regrouping terms on the right hand side, we obtain 
\begin{multline}
	 	 0
	 = (n-2)\bigl(\de^g(\cn{g}\th) - (\de^g \th) \th + \cn{g}_T \th)  +(2-n){\rm d}\de^g\th  + \frac{(n-2)(n-3)}2{\rm d}(\|\th\|_g^2)\\+ \frac{n-2}{2n} {\rm d}\scal{\cn{}}_g - \frac{(n-2)}{2}\de^g {\rm d}\th.
	 \end{multline}
	Multiplying this equation by $2/(n-2)$ yields \eqref{eq:gauduchonformel}.
\end{proof}

The next two results refer to the Gauduchon metric --- see Theorem~\ref{th:gaudmet} --- associated to Weyl structures on compact manifolds. 

\begin{lem}\cite[Théorème 2(ii)]{Ga95} 
\label{lem:threerels}
Let $\cn{}$ be a Weyl-Einstein structure on a connected compact conformal manifold $(M,c)$ which is oriented and of dimension $n>2$, and let $g$ be the Gauduchon metric in $c$. 
There exists a constant $K$ such that
\begin{equation}\label{eq:threerel1}
	\begin{aligned}
		\scal{g} - (n+2)\|\th\|_g^2 = K \\
	\end{aligned}
\end{equation}
and
\begin{equation}\label{eq:threerel2}
	\begin{aligned}
		\scal{\cn{}}_g + n(n-4)\|\th\|_g^2 = K.
	\end{aligned}
\end{equation}
Furthermore, 
\begin{equation}\label{eq:threerel3}
	\begin{aligned}
		\De^g \th = \frac{2}{n}\scal{\cn{}}_g \th.
	\end{aligned}
\end{equation}
\end{lem}
\begin{proof} Since the co-differential of the Lee form vanishes,  (\ref{eq:gauduchonformel}) becomes
\begin{equation}\label{eq:gauduchonformelsimpl}
	  \frac{1}{n}{\rm d}\scal{\cn{}}_g  + 2\de^g(\cn{g}\th) - \de^g {\rm d}\th + 2\cn{g}_T \th + (n-3){\rm d}(\|\th\|_g^2)= 0.
\end{equation}
Let ${\rm d}v$ denote the natural volume form on $(M,g)$. Since $\de^g$ is the formal adjoint of ${\rm d}$ for the global inner product induced by $g$ on the bundle of differential forms and $\de^g\th = 0$, it follows
\begin{equation}
	\begin{aligned}
		\int_M \langle {\rm d}(\|\th\|_g^2), \th \rangle {\rm d}v = \int_M \langle {\rm d}\scal{\cn{}}_g, \th \rangle {\rm d}v = 0.
	\end{aligned}
\end{equation}
Furthermore, the Koszul-formula yields
\begin{eqnarray}
		\langle \cn{g}_T \th, \th \rangle &= &\cn{g}_T \th(T) = T(g(T,T)) - g(T,\cn{g}_T T) =\frac{1}{2}{\rm d}(\|\th\|_g^2)(T) \\
		&= & \frac{1}{2}\langle {\rm d}(\|\th\|_g^2), \th \rangle,
\end{eqnarray}
thus,
\begin{equation}
	\begin{aligned}
	 \int_M \langle \cn{g}_T \th, \th \rangle {\rm d}v = \int_M \langle {\rm d}(\|\th\|_g^2), \th \rangle {\rm d}v = 0.
	\end{aligned}
\end{equation}
Therefore (\ref{eq:gauduchonformelsimpl}) yields 
\begin{equation}\label{eq:deln}
	 \int_M \langle \de^g\cn{g} \th - \frac{1}{2}\de^g {\rm d}\th, \th \rangle {\rm d}v = 0.
\end{equation}

Recall that the symmetric and skew-symmetric parts of $\nabla^g\theta$ are given, respectively, by the Lie derivative of the metric with respect to $T$, and the exterior differential of $\theta$; namely, 
\begin{equation}\label{eq:nabss}
\nabla^g\th=\frac12\mathcal L_Tg+\frac12 {\rm d}\th.
\end{equation}
By \eqref{eq:codif}, if $\{e_i\}_{i=1}^n$ denotes a local orthonormal frame, \eqref{eq:deln} becomes
\begin{equation}\label{eq:globscals}
	\begin{aligned}
	0&=\int_M \langle \de^g(\cn{g} \th - \frac{1}{2}{\rm d}\th), \th \rangle {\rm d}v  = \frac12\int_M \langle \de^g\mathcal L_Tg, \th \rangle\, {\rm d}v \\
	&= -\frac12 \sum_{i=1}^n \int_M \langle e_i \lrcorner (\cn{g}_{e_i}\mathcal L_Tg), \th \rangle\, {\rm d}v \\
	&= - \frac12\sum_{i=1}^n \int_M \left\{
	 e_i(\langle e_i \lrcorner (\mathcal L_Tg), \th \rangle) -  \langle (\cn{g}_{e_i}e_i) \lrcorner (\mathcal L_Tg), \th \rangle \right.\\&\qquad\qquad\left.- \langle e_i \lrcorner (\mathcal L_Tg), \cn{g}_{e_i}\th \rangle\, \right\} {\rm d}v.
	\end{aligned}
\end{equation}
Now define the $1$-form $\ps$ by $\ps(X) = \langle X\lrcorner( \nabla^g\mathcal L_Tg), \th \rangle$ for all $X\in\Xf(M)$. Then, since $M$ is compact, the integral of $\de^g \ps$ on $M$ vanishes. Explicitly, by \eqref{eq:codif},\begin{equation} \label{eq:intMcod}
	\begin{aligned}
0=\int_M	 (\delta^g \psi) {\rm d}v= -\int_M\left\{ \sum_{i=1}^n e_i(\langle e_i \lrcorner (\mathcal L_Tg), \th \rangle) -  \langle (\cn{g}_{e_i}e_i) \lrcorner (\mathcal L_Tg), \th \rangle\right\}{\rm d}v.
	\end{aligned}
\end{equation}
 Moreover, from \eqref{eq:nabss} we further get
\begin{equation} \label{eq:3t}
	\begin{aligned}
\int_M \langle e_i \lrcorner (\mathcal L_Tg), \cn{g}_{e_i}\th \rangle\, {\rm d}v &= \int_M \left\{\|e_i \lrcorner (\mathcal L_Tg)\|^2 + \frac{1}{2} \langle e_i \lrcorner (\mathcal L_Tg), e_i \lrcorner {\rm d}\th \rangle\,\right\} {\rm d}v \\
	&= \int_M \|e_i \lrcorner (\mathcal L_Tg)\|^2 {\rm d}v.
	\end{aligned}
\end{equation}
Therefore, (\ref{eq:globscals}), \eqref{eq:intMcod} and \eqref{eq:3t} yield $\sum_{i=0}^n \int_M \|e_i \lrcorner (\mathcal L_Tg)\|^2{\rm d}v = 0$. This implies  $\|e_i \lrcorner (\mathcal L_Tg)\|^2 = 0$ for all $1\leq i \leq n$. We thus conclude from \eqref{eq:nabss}
\begin{equation}\label{eq:killing}
	\lie_T g = 0,\qquad	\cn{g} \th = \frac{1}{2}{\rm d}\th.
\end{equation}
In particular $T$ is a Killing vector field for $g$ and
\begin{equation}\label{eq:killingT}
		\cn{g}_T \th = -\frac{1}{2}{\rm d}(\|\th\|_g^2). 
\end{equation}
due to the Koszul formula and the definition of ${\rm d}\th$.

Now, using (\ref{eq:killing}) and $\delta^g\theta=0$, (\ref{eq:ricrel})  reads
\begin{equation}
    \ric{\cn{}} = \ric{\cn{g}} - (n-2)\bigl(\frac{1}{2}{\rm d}\th - \th \otimes \th\bigr)  - (n-2)\|\th\|_g^2\, g, 
\end{equation}
and (\ref{eq:scalrel}) becomes
\begin{equation}\label{eq:scalG}
    \scal{\cn{}}_{g} = \scal{g} - (n-1)(n-2)\|\th\|_g^2.
\end{equation}
Since $\cn{}$ satisfies the Weyl-Einstein condition (\ref{eq:EW}) we thus get 
\begin{equation}\label{eq:ricG}
    \ric{\cn{g}} = \frac{1}{n}\scal{g} g + (n-2)\bigl(\frac{1}{n}\|\th\|_g^2 g - \th \otimes \th\bigr).
\end{equation}Taking the co-differential in this equation and using \eqref{eq:codiff1}, \eqref{eq:codiff2} and \eqref{eq:killingT}, we get
\begin{equation}
	\begin{aligned}
			\de^g\ric{\cn{g}} &= -\frac{1}{n}{\rm d}\scal{g} - (n-2)\bigl(\frac{1}{n}{\rm d}(\|\th\|_g^2) + (\de^g \th)\th - \cn{g}_T \th)\\
			&= -\frac{1}{n}{\rm d}\scal{g} - (n-2)\bigl(\frac{1}{n}{\rm d}(\|\th\|_g^2) - \cn{g}_T \th) \\
			&= -\frac{1}{n}{\rm d}\scal{g} -\frac{ (n-2)(n+2)}2{\rm d}(\|\th\|_g^2) ,
	\end{aligned}
\end{equation}By \eqref{eq:bianccontr}, this equation equals $-\frac12 {\rm d}\scal{g}$ so we finally obtain
\begin{equation}
	\begin{aligned}
		 {\rm d}\scal{g} - (n+2){\rm d}(\|\th\|_g^2) = 0.
	\end{aligned}
\end{equation}
Since $M$ is connected, there exists a constant $K$ such that
\begin{equation}
	\begin{aligned}
		 \scal{g} - (n+2)\|\th\|_g^2 = K,
	\end{aligned}
\end{equation}
and by (\ref{eq:scalG})
\begin{equation}\label{eq:scalneg}
	\begin{aligned}
		 \scal{\cn{}}_g + n(n-4)\|\th\|_g^2 = K.
	\end{aligned}
\end{equation}
Now only (\ref{eq:threerel3}) remains to be proven. Since $\de^g\th = 0$, we have
\begin{equation}
	\De^g \th = \de^g {\rm d}\th.
\end{equation}
On the one hand, (\ref{eq:WF}), \eqref{eq:ricG} and \eqref{eq:scalG} imply
\begin{equation}\label{eq:Lth}
	\begin{aligned}
	\De^g\th &=  (\cn{g})^*\cn{g}\th  + \ric{\cn{g}}(\th) \\
	&= (\cn{g})^*\cn{g}\th + \frac{1}{n}\scal{g} \th + (n-2)(\frac{1}{n}\|\th\|^2_g \th - \|\th\|^2_g \th) \\
	&= (\cn{g})^*\cn{g}\th + \frac{1}{n}(\scal{g} \th - (n-2)(n-1)\|\th\|^2_g \th)\\
	&= (\cn{g})^*\cn{g}\th + \frac{1}{n}\scal{\cn{}}_g \th.
	\end{aligned}
\end{equation}
On the other hand, (\ref{eq:killing}) implies
\begin{equation*}
	\begin{aligned}
		(\cn{g})^*\cn{g}\th &= - \sum_{i=1}^n \cn{g}_{e_i} \cn{g}_{e_i} \th \\
		&= - \frac{1}{2}\sum_{i=1}^n \cn{g}_{e_i}(e_i \lrcorner {\rm d}\th) \\
		&= - \frac{1}{2}\de^g {\rm d}\th = - \frac{1}{2}\De^g \th.
	\end{aligned}
\end{equation*}
This equation together with \eqref{eq:Lth} finally yields \eqref{eq:threerel3}.
\end{proof}
The last point of Lemma~\ref{lem:threerels} yields the following important result.
\begin{lem}\cite[Théorème 2(iii)]{Ga95}\label{lem:gauduchonexact}
Let $\cn{}$ be a Weyl-Einstein structure on a connected compact conformal manifold $(M,c)$ which is oriented and of dimension $n>2$, and let $g$ be the Gauduchon metric in $c$.  If $\scal{\cn{}}_g < 0$, then $\cn{} = \cn{g}$ and $g$ is an Einstein metric.
\end{lem}
\begin{proof}
Let ${\rm d}v$ denote the volume form associated to $g$. Since $\de^g \th = 0$, equation (\ref{eq:threerel3}) becomes 
\begin{equation}
	\begin{aligned}
		\de^g {\rm d}\th = \frac{2}{n}\scal{\cn{}}_g \th.
	\end{aligned}
\end{equation}
Hence, taking the global inner product of this equation with $\th$ yields
\begin{equation}
	\begin{aligned}
		\frac{2}{n}\int_M \scal{\cn{}}_g \|\th\|_g^2 {\rm d}v=\int_M \langle \de^g {\rm d}\th, \th \rangle {\rm d}v= \int_M \|{\rm d}\th\|^2_g {\rm d}v \geq 0
	\end{aligned}
\end{equation}
since ${\rm d}$ is the formal adjoint of $\de^g$. This equation implies that if $\scal{\cn{}}_g < 0$, then $\|\th\|_g^2 = 0$. It follows that  $\nabla=\nabla^g$ and, by \eqref{eq:EW}, that $g$ is an Einstein metric.
\end{proof}

\section{Left-invariant Weyl structures on conformal Lie groups}\label{s:nilman}\label{s:conform}

In this section we describe geometrical features of Riemannian Lie groups and we extend them to left-invariant conformal and Weyl structures.

Let $G$ be a connected Lie group of dimension $n$ endowed with a left-invariant metric $g$ and let $\cn{g}$ denote the Levi-Civita connection of $(G,g)$; we call $(G,g)$ a Riemannian Lie group. Let $\gf$ be the Lie algebra of $G$ and consider the inner product induced by the left-invariant metric $g$ on $\gf$, which will be also denoted by $g$. By abuse of language, we may say that $\nabla^g$ is the Levi Civita connection of $(\mg,g)$, and similarly for the curvature tensors.

The covariant derivative of a left-invariant vector field with respect to another left-invariant vector field is again left-invariant. Identifying left-invariant vector fields $X$ on $G$ with their values $x$ at $\mg$, the Koszul's formula reads
\begin{equation}\label{eq:Koszullie}
    \cn{g}_x y = \frac{1}{2}\bigl([x,y] - \ad_y^*x - \ad_x^*{y}\bigr), \quad \mbox{ for all }x,y\in\mg,
\end{equation} 
where $\ad_x^*$ denotes the $g$-adjoint operator of $\ad_x$.

Any curvature tensor on $G$ is constant when evaluated on left-invariant vector fields and thus it is determined by its value on elements in $\mg$. Hence curvature tensors of $(G,g)$ will be treated as algebraic tensors on $\mg$. In particular, one has the following expression for the Ricci-tensor \cite[\S 7.38]{Besse2008}
\begin{equation}\label{eq:PBZ}
\Ric^{\nabla^g}=P-\frac12 B-\frac12 (\ad_z+\ad_z^*),
\end{equation}
where $B$ denotes the Cartan-Killing form of $\mg$, $z\in\mg$ is the unique element such that $g(z,x)=\tr \ad_x$ for all $x\in\mg$, and $P$ is defined by means of an orthonormal basis $\{e_i\}_{i=1}^n$ of $\mg$ as
\begin{equation}
	P(x,y):= \sum_{i,k} -\tfrac12 g(e_i,[x,e_k])g(e_i,[y,e_k]) + \tfrac14 g(x,[e_i,e_k])g(y,[e_i,e_k]), \quad \mbox{ for all }x,y\in\mg.
\label{eq:Ric4}
\end{equation}

We will be interested in conformal structures and Weyl structures on Lie groups that are invariant under left-translations, in the following sense.

\begin{defi}A conformal Lie group $(G,c)$ is a Lie group $G$ endowed with a conformal class $c$ for which there exists a left-invariant representative $g\in c$. In particular, $(G,g)$ is a Riemannian Lie group.
\end{defi}
Notice that a left-invariant representative in a conformal class $c$ is unique up to constant rescaling.

\begin{defi}A Weyl structure on a conformal Lie group $(G,c)$ is called left-invariant if its Lee form with respect to any left-invariant metric $g\in c$ is left-invariant.
\end{defi}

It is clear from Section \ref{ch:prelim}, that on a conformal manifold $(M,c)$, exact Weyl structures on $(M,c)$ are the Levi-Civita connections of the Riemannian metrics in $c$. In the context of conformal Lie groups $(G,c)$, exact left-invariant Weyl structures correspond to Riemannian metrics in $c$ for which left-translations are homotheties.

Given a conformal Lie group $(G,c)$ and fixing a left-invariant metric $g\in c$, left-invariant Weyl structures on $(G,c)$ are in one-to-one correspondence with left-invariant 1-forms on $G$, i.e. elements in $\mg^*$, through formula \eqref{eq:weylrel}.
Notice that, having fixed the left-invariant metric $g\in c$ (unique up to constant rescaling), one can identify  elements in $\mg^*$ with vectors in $\mg$ via $g$. This identification is used in the following lemma to write the Weyl-Einstein condition of a left-invariant Weyl structure in terms of its  Lee form $\theta$  and the Ricci tensor of $g$.

\begin{lem} \label{lem:Ricgen} Let $(G,c)$ be a conformal Lie group and let $g\in c$ be a left-invariant metric. The Weyl structure $\nabla$ induced by $\th\in\mg^*$ (identified with its metric dual $\theta\in \mg$) is Weyl-Einstein if and only if the following equation holds
\begin{equation}\label{eq:Ricgen}
\ric{\nabla^g}=\frac1{n}\left(\scal{g}  +(n-2)(\tr\ad_\th+\|\th\|_g^2)\right)g-(n-2)\left[\ad_\th^{\rm{sym}}+\theta\otimes\th\right],
\end{equation}
where $\ad_\th^{\rm{sym}}=(\ad_\th+\ad_\th^*)/2$ is the symmetric part of $\ad_\th$.
\end{lem}
\begin{proof}
Using \eqref{eq:Koszullie} we get, for every $x\in\mg$,
\begin{equation}
g(\cn{g}_x \th,x) = \frac{1}{2}g\bigl([x,\th] - \ad_\th^*x - \ad_x^*\th,x\bigr)=g([x,\th],x),
\label{eq:nabth}
\end{equation} 
whence $(\cn{g} \th)^{\rm{sym}}=-\ad_\th^{\rm{sym}}$ and by taking the trace, $\delta^g\th=\tr\ad_\th$. The lemma thus follows directly from \eqref{eq:ricrelwe}.
\end{proof}

\section{Weyl-Einstein structures on compact solvmanifolds}\label{ch:solv}

We are now ready to study the question of the existence of Weyl-Einstein structures on compact conformal solvmanifolds.

Consider a solvable Lie group $G$, which admits a discrete subgroup $\Gamma$ such that the quotient manifold $M\coloneqq\Gamma\backslash G$ is compact. In particular $G$ is unimodular \cite{Mil76}.

Any left-invariant metric $g$ on $G$ defines a Riemannian metric $\bar g$ on the quotient manifold $M$, in a way that the natural projection $\pi:(G,g)\to (M,\bar g)$ is a local isometry. We shall denote by $\bar c$ the conformal class of $\bar g$ on $M$.

We say that $(M,\bar g)$ is a Riemannian solvmanifold and that $(M,\bar c)$ is a conformal solvmanifold, both defined by $(G,g)$. In the particular case where $G$ is nilpotent, $(M,\bar g)$ will be called nilmanifold instead of solvmanifold.

The goal of this section is to show that conformal solvmanifolds admitting Weyl-Einstein structures  arise only from unimodular solvable Lie groups admitting Einstein metrics.

\begin{theorem}\label{th:WEiffE}
Let $(M,\bar c)$ be a conformal solvmanifold defined by $(G,g)$. Then $(M,\bar c)$ admits a Weyl-Einstein structure if and only if $(G,g)$ is Einstein.
\end{theorem}
\begin{proof}

Suppose that $(M,\bar c)$ admits a Weyl-Einstein structure $\nabla$. If $(G,g)$ is flat, then it is in particular Einstein so there is nothing to prove. 

Assume that $(G,g)$ is not flat. By \cite[Theorem 3.1]{Mil76}, $(G,g)$ has (constant) strictly negative scalar curvature:
\begin{equation}\label{eq:scalgneg}
\scal{g} <0.
\end{equation}

Let $\bar g\in \bar c$ denote the metric on $M$ induced by $g$. We denote by $g_0\in\bar c$ the Gauduchon metric associated to $\cn{}$, according to Theorem \ref{th:gaudmet}. We shall first prove that $g_0$ is actually an Einstein metric.

Since $\bar g,g_0\in \bar c$, there exists $f \in \Cc^{\infty}(M)$ such that $\bar g = e^{2f}g_0$.  The scalar curvatures of $\bar g$ and $g_0$ are linked by the formula (cf. \cite[\S 1.159]{Besse2008})
\begin{equation}\label{eq:scalrel2}
	\begin{aligned}
		 \scal{\bar g} = e^{-2f}\bigl(\scal{g_0} + 2(n-1)\De^{g_0}f - (n-1)(n-2)\|{\rm d}f\|_{g_0}^2\bigr).
	\end{aligned}
\end{equation}
Hence, if $K$ denotes the constant in (\ref{eq:threerel1}), we have
\begin{equation}
	\begin{aligned}\label{eq:transition}
		 \scal{\bar g} = e^{-2f}\bigl(K + (n+2)\|\th\|_{g_0}^2 + 2(n-1)\De^{g_0}f - (n-1)(n-2)\|{\rm d}f\|_{g_0}^2\bigr).
	\end{aligned}
\end{equation}
Notice that $\scal{\bar g}$ is a negative constant because it is induced by $g$, which satisfies \eqref{eq:scalgneg}.

Since $M$ is compact, $f$ has a maximum at some point $x_0 \in M$, and of course $\|{\rm d}f\|^2_{g_0}(x_0)=0$. Evaluating (\ref{eq:transition}) at $x_0$ yields
\begin{equation}
	\begin{aligned}
		 K = e^{2f(x_0)}\scal{\bar g} - (n+2)\|\th\|_{g_0}^2(x_0) - 2(n-1)\De^{g_0}f(x_0).
	\end{aligned}
\end{equation}
We know that $\scal{\bar g} < 0$ and $\De^{g_0}f(x_0) \geq 0$, since $x_0$ is a maximum of $f$. Hence, $K < 0$ and thus, by (\ref{eq:threerel2}), we conclude
\begin{equation}
	\scal{\cn{}}_{g_0} < 0.
\end{equation}
Therefore, Lemma~\ref{lem:gauduchonexact} implies that the Weyl-Einstein structure $\cn{}$ is exact and $g_0$ is an Einstein metric with Levi-Civita connection $\cn{}$. Recall that any Einstein metric has constant scalar curvature, so in particular $\scal{g_0}$ is constant.

To finish the proof, we will show that $\bar g$ must be Einstein as well.

 Again, the compactness of $M$ allows us to find $x_1 \in M$ such that $f$ has a minimum in $x_1$. If $x_0$ is as above, we know that $\|{\rm d}f\|^2_{g_0}(x_0) = \|{\rm d}f\|^2_{g_0}(x_1) = 0$ and $ \De^{g_0} f(x_1)\leq 0\leq \De^{g_0} f(x_0) $. 
These facts together with \eqref{eq:scalrel2} and $\scal{\bar g} < 0$ imply
\begin{equation}\label{eq:scl}
    \begin{aligned}
        \scal{g_0} + 2(n-1)\De^{g_0} f(x_1) \leq \scal{g_0} + 2(n-1)\De^{g_0} f (x_0) < 0.
    \end{aligned}
\end{equation}
In addition, $e^{-2f(x_0)} \leq e^{-2f(x_1)}$ so, being both terms in \eqref{eq:scl} negative, and using \eqref{eq:scalrel2}, we get
\begin{eqnarray*}
 \scal{\bar g}&=&   e^{-2f(x_1)}\bigl(\scal{g_0} + 2(n-1)\De^{g_0} f(x_1)\bigr)\\ &\leq&
 e^{-2f(x_0)}\bigl(\scal{g_0} + 2(n-1)\De^{g_0} f(x_1)\bigr)\\ &\leq& 
   e^{-2f(x_0)}\bigl(\scal{g_0} + 2(n-1)\De^{g_0} f (x_0)\bigr)=\scal{\bar g}.
\end{eqnarray*}
This implies already $f(x_0) =f (x_1)$ and thus $f$ is constant. Hence $\bar g$, being a constant multiple of $g_0$, turns out to be an Einstein metric as well. Therefore, $g=\pi^*\bar g$ is a left-invariant Einstein metric since $\pi$ is a local isometry.

The converse statement is trivial.
\end{proof}

Notice that a solvable unimodular Riemannian Lie group is Einstein if and only if it is flat. Indeed, any solvable Riemannian Lie group has constant non-positive scalar curvature \cite[Theorem 3.1]{Mil76}. Thus, if it is Einstein, the Ricci tensor is a non-positive multiple of the identity. However, by \cite[Corollary 3.3]{Do82} it cannot be a strictly negative multiple of the identity. Hence the solvable Riemannian Lie group must be Ricci-flat and hence flat by \cite{AlKi75}.

Consequently, one can restate the previous result as follows:

\begin{co}\label{cor:42}
Let $(M,\bar c)$ be a conformal solvmanifold defined by $(G,g)$.  Then $(M,\bar c)$ admits a Weyl-Einstein structure if and only if $(G,g)$ is flat.
\end{co}

Recall that left-invariant metrics on non-abelian nilpotent Lie groups cannot be flat \cite{Mil76}. So we can state a second consequence of Theorem \ref{th:WEiffE}.

\begin{co}A conformal nilmanifold admits Weyl-Einstein structures if and only if it is a flat torus.
\end{co}

\section{Weyl-Einstein structures on nilpotent Lie groups}\label{ch:einstein}

In this section we drop the compactness assumption, and turn our attention to left-invariant Weyl structures on nilpotent conformal  Lie groups. We show that conformal nilpotent non-abelian Lie groups do not admit left-invariant Weyl-Einstein structures. 

Let $G$ be a nilpotent non-abelian Lie group of dimension $n\geq 3$ with Lie algebra $\mg$. Let $\mg'$ denote the commutator of $\mg$, namely $\mg'=[\mg,\mg]$, and set $\mg'':=[\mg,\mg']$. Since $\mg$ is nilpotent and non-abelian, one has $0\neq \mg'\supsetneq \mg''$. In particular, this implies that $\dim (\mg')^\bot=\dim \mg-\dim \mg'\geq 2$. Indeed, if one supposes that $\mg=\R x\oplus \mg'$ for some $x\in\mg$, then 
\[
\mg'=[\mg,\mg]=[\R x\oplus \mg',\R x\oplus \mg']\subset \mg '',
\]leading to a contradiction. Hence $\dim (\mg')^\bot\geq 2$.

Notice that for any $x\in\mg$, $\ad_x$ is a nilpotent endomorphism because $\mg$ is a nilpotent Lie algebra. This implies that both the element $z$ and the Killing form $B$ of $\mg$ appearing in \eqref{eq:PBZ} vanish. Therefore, $\ric{\nabla^g}=P$, where $P$ is given in \eqref{eq:Ric4}, and one can easily show that $\mg$ is not Einstein, if it is not abelian. 

In fact, since $\mg$ is nilpotent, there exist nonzero elements $x\in (\mg')^\bot$, for which the second term of $\ric{\nabla^g}(x,x)$ in \eqref{eq:Ric4} vanishes and thus $\ric{\nabla^g}(x,x)\leq0$. If in addition $\mg$ is not abelian, the center $\mz$ of $\mg$ has non-trivial intersection with the commutator $\mg'$. Hence, for every $0\neq y\in\mz\cap\mg'$, the first term of $\ric{\nabla^g}(y,y)$ in \eqref{eq:Ric4} vanishes, yielding $\ric{\nabla^g}(y,y)>0$. So $\ric{\nabla^g}$ cannot be a multiple of the metric and thus $(G,g)$ is not Einstein.

Using similar arguments while working with the Weyl-Einstein condition as described in Lemma \ref{lem:Ricgen}, we get the following:

\begin{pr} \label{pro:nilWE} Let $G$ be a non-abelian nilpotent Lie group. For any left-invariant conformal structure $(G,c)$, and any left-invariant Weyl structure $\nabla$ on it, $\nabla$ is not Weyl-Einstein.
\end{pr}

\begin{proof} Assume for a contradiction that $\nabla$ is a left-invariant Weyl-Einstein structure on the conformal Lie group $(G,c)$, and let  $\theta$ denote the left-invariant Lee form of $\nabla$ corresponding to a left-invariant metric $g\in c$. Notice that $\tr\ad_\th=0$ because $\mg$ is nilpotent (here we identify $\th\in \mg^*$ with its metric dual $\th\in\mg$ via $g$). Hence, by Lemma \ref{lem:Ricgen}, for any $x\in\mg$ we have
\begin{eqnarray}\label{eq:ig4}
\ric{\nabla^g}(x,x)
=\frac1{n}\left(\scal{g}  +(n-2)\|\th\|_g^2\right) \|x\|^2_g+(n-2)\left[g([x,\theta],x) -g(\theta,x)^2\right]
.
\end{eqnarray}

Since $\dim (\gf')^\bot\geq 2$, there exists a unit vector $x\in(\gf'+\R \theta)^\bot$, for which \eqref{eq:ig4} becomes
\begin{eqnarray}\nonumber
\ric{\nabla^g}(x,x)
&=&\frac1{n}\left(\scal{g}  +(n-2)\|\th\|_g^2\right).
\end{eqnarray}Moreover, from \eqref{eq:Ric4} we can easily see that $\ric{\nabla^g}(x,x)=P(x,x)\leq 0$ because $x\bot\mg'$. Therefore, we get
\begin{equation}\label{eq:des}
\scal{g}  +(n-2)\|\th\|_g^2\leq 0.
\end{equation}

In addition, for any unit vector $y\in\mathfrak z\cap \gf'$, \eqref{eq:ig4} gives
\[
\ric{\nabla^g}(y,y)+(n-2)\theta(y)^2=\frac1{n}\left(\scal{g}  +(n-2)\|\th\|_g^2\right) 
,\] but \eqref{eq:Ric4} gives $\ric{\nabla^g}(y,y)=P(y,y)>0$. So we get $\scal{g}  +(n-2)\|\th\|_g^2 >0$, contradicting \eqref{eq:des}.
\end{proof}

\section{Weyl-Einstein structures on almost abelian solvable Lie groups}\label{ch:aa}

In order to construct non-trivial examples of left-invariant Weyl-Einstein structures on conformal Lie groups, we will now consider the more general setting of solvable Lie groups. However, since in full generality the problem is too involved, we will restrict our attention to the almost abelian case.

Recall that a Lie algebra  is called almost abelian if it has an abelian ideal of codimension 1 and, accordingly, any Lie group whose Lie algebra is almost abelian, is called an almost abelian Lie group. Of course, almost abelian Lie algebras are 2-step solvable.

Let $(G,c)$ be an almost abelian conformal Lie group of dimension $n\geq 3$ and $g\in c$ a left-invariant metric. Let $\mg$ denote the Lie algebra of $G$ and let $\mh$ be a codimension 1 abelian ideal in $\mg$. Notice that $\mh$ is unique, unless $\mg$ is abelian or isomorphic to a direct sum of an abelian Lie algebra and the Heisenberg Lie algebra of dimension 3 (cf. proof of \cite[Proposition 1]{Fr12}). 

Using the chosen metric $g$ on $\mg$, consider a unit vector $b\in \mh^\bot$, so that $\mg$ decomposes as the orthogonal direct sum $\mg=\R b\oplus \mh$. Since $\mh$ is an abelian ideal, the Lie algebra structure of $\mg$ is determined by the adjoint map $\ad_b$, which preserves $\mh$. In particular, $\mg$ is nilpotent if and only if $\ad_b$ is a nilpotent endomorphism of $\mh$.

Let $A$ and $S$ be the skew-symmetric and symmetric parts (with respect to $g$) of $\ad_b|_\mh:\mh\to\mh$, respectively. We will say that $A$ and $S$ are the endomorphisms associated to $(\mg,g)$, understanding that they are defined up to the choice of a unit vector $b$ whose orthogonal is the abelian ideal $\mh$.

The Levi Civita connection of $(\mg,g)$ is given by \eqref{eq:Koszullie} and verifies
\begin{equation}
\label{eq:LCaa}
\nabla^g_b b=0,\quad \nabla^g_bu=Au,\quad \nabla^g_ub=-Su,\quad \nabla^g_uv=\lela Su,v\rira b,\qquad \forall u,v\in\mh.
\end{equation}Using these formulas, it is straightforward that the curvature tensor \eqref{eq:riemc} verifies,  for $u,v,w\in\mh$, 
\begin{eqnarray}
R^g(u,v)w&=&
\lela Sv,w\rira Su-\lela Su,w\rira Sv,\label{eq:Raa}\\
R^g(u,v)b&=&0,\nonumber\\
R^g(b,v)w
&=&\lela (S^2- [A,S])v,w\rira b,\nonumber\\
R^g(b,v)b&=&
[A,S]v-S^2 v.\nonumber
\end{eqnarray}
Therefore, one can easily check that the Ricci tensor takes the form 
\begin{equation}\label{eq:ricaa}
\ric{\nabla^g}=-\tr (S^2)b\otimes b+[A,S]-\tr S \cdot S.
\end{equation} In particular $\ric{\nabla^g}$ preserves $\R b$ and $\mh$, and the scalar curvature of $(\mg,g)$ is 
\begin{equation}
\label{eq:scalaa}
\scal{g}=-\tr(S^2)-(\tr S)^2.
\end{equation}

Notice that when $S=k\Id_{\mh}$ for some $k\in\R$, then \eqref{eq:ricaa} implies that $\ric{\nabla^g}=-k^2(n-1)g$, and thus $(G,g)$ is Einstein. In fact, the converse also holds:

\begin{lem}\label{re:Sid}
If $(G,g)$ is Einstein, then $S$ is a multiple of the identity. 
\end{lem}

\begin{proof} Indeed, if $\ric{\nabla^g}=\lambda g$ for some $\lambda\in \R$, then contracting Equation \eqref{eq:ricaa} with $b$ twice, we get $\lambda=-\tr(S^2)$. Moreover, projecting Equation \eqref{eq:ricaa} onto $\mh\otimes \mh$  and using the Einstein condition together with $\lambda=-\tr(S^2)$, we further obtain 
\begin{eqnarray}\label{eq:bb}
-\tr(S^2) \, g|_\mh &=&[A,S]-\tr S\cdot S.
\end{eqnarray}
Taking trace on both sides, this equation implies
\begin{equation}\label{eq:tt}
\tr(S^2)(n-1)=(\tr S)^2.
\end{equation}
Furthermore, writing  $S=\frac{\tr S}{(n-1)}\id_\mh+S_0$, where $S_0$ is trace-free, and using  \eqref{eq:tt},  \eqref{eq:bb} becomes
\begin{eqnarray}
\, [A,S_0]&=&\tr S\cdot S_0.
\end{eqnarray}
This implies that either $\tr S=0$ or $S_0=0$. The latter clearly yields $S=k \Id_\mh$ for some $k\in\R$. Since $n>2$, the former case together with \eqref{eq:tt} imply $\tr(S^2)=0$ and therefore $S=0$. 
\end{proof}

The next theorem generalizes Lemma \ref{re:Sid} and gives necessary and sufficient conditions for a conformal almost abelian Lie group $(G,c)$ to admit a left-invariant Weyl-Einstein structure, in terms of the endomorphisms $A$ and $S$ associated to $(\mg,g)$.

\begin{theorem}\label{pro:aa} Let $G$ be an almost abelian Lie group of dimension $n\geq 3$ with Lie algebra $\mg=\R b\oplus \mh$. Let $g$ be a metric on $\mg$ with associated endomorphisms $A,S$ and let $c:=[g]$ be the conformal structure induced by $g$.
Let $\nabla$ be a left-invariant Weyl structure on $(G,c)$ and let $\th\in\mg^*$ be the (left-invariant) Lee form of $\nabla$ with respect to $g$.
Then $\nabla$ is a Weyl-Einstein structure if and only if one of the following exclusive conditions holds: 
\begin{enumerate}
\item \label{it:id} $S=k\Id_{\mh}$ for some $k\in\R$, and $\theta=0$ or $\theta=kb$, or
\item \label{it:S0} $S\neq 0$, $(\tr S)^2=\tr(S^2)(n-2)$, $[A,S]=0$ and $\theta=\frac{\tr(S)}{(n-2)}b$.
\end{enumerate}
\end{theorem}

\begin{proof} Assume that $\nabla$ is Weyl-Einstein.
Consider the orthogonal decomposition $\mg=\R b \oplus \mh$ and write accordingly $\th=\mu b+v$, with $v\in \mh$ and $\mu\in\R$.
We readily compute
$$\ad_\th=\mu S-b\otimes Sv,\qquad \tr\ad_\th=\mu\tr S.$$ Replacing these terms in \eqref{eq:Ricgen}, and using \eqref{eq:ricaa} and \eqref{eq:scalaa}, we conclude that 
%the Weyl structure defined by $\th=\mu b+v$ is Weyl-Einstein if and only if 
the following equation holds
\begin{multline}\label{eq:a1}
-\tr (S^2)b\otimes b+[A,S]-\tr S \cdot S
=\frac1{n}\left(-\tr(S^2)-(\tr S)^2 +(n-2)(\mu\tr S+\mu^2+\| v\|_g^2)\right) \Id \\-(n-2)\left[\mu S-\frac12(b\otimes Sv+Sv\otimes b) +\mu^2 b\otimes b+\mu b\otimes v+\mu v\otimes b+v\otimes v\right].
\end{multline}

We first show that \eqref{eq:a1} implies $v=0$. In fact, contracting this equation with $b$ yields 
\begin{eqnarray*}-\tr (S^2)b&=&\frac1{n}\left(-\tr(S^2)-(\tr S)^2 +(n-2)(\mu\tr S+\mu^2+\| v\|_g^2)\right)b\\&&-(n-2)\left[-\frac12Sv+\mu^2b+\mu v\right].\end{eqnarray*}
Since $b$ is orthogonal to $\mh$ and $v,Sv\in\mh$, this is equivalent to the system
\begin{equation}\label{a2}Sv=2\mu v
\end{equation}
\begin{equation}\label{a3}(n-2)\mu^2-\tr (S^2)=\frac1{n}\left(-\tr(S^2)-(\tr S)^2 +(n-2)(\mu\tr S+\mu^2+\| v\|_g^2)\right).
%(n-1)\tr(S^2)=(n-1)(n-2)\mu^2+(\tr S)^2-(n-2)(\mu\tr S+\| v\|_g^2).
\end{equation}
Note that this last equation simplifies to 
\begin{equation}\label{a6}(n-1)\tr(S^2)=(n-1)(n-2)\mu^2+(\tr S)^2-(n-2)(\mu\tr S+\| v\|_g^2).
\end{equation}

Reinjecting \eqref{a3} in \eqref{eq:a1} and projecting onto $\mh\otimes \mh$ yields
\begin{equation}\label{eq:a4}
[A,S]-\tr S \cdot S
=((n-2)\mu^2-\tr (S^2))\Id_\mh-(n-2)(\mu S+v\otimes v).
\end{equation}

Contracting this last equation with $v$ and using \eqref{a2} gives
\begin{equation}\label{eq:a5}
[A,S]v-2\mu(\tr S) v
=((n-2)\mu^2-\tr (S^2))v-(n-2)(2\mu^2v+\|v\|^2 v).
\end{equation}

Assume now for a contradiction that $v\ne 0$. Clearly $g([A,S]v,v)=0$, so \eqref{eq:a1} yields
\begin{equation}\label{a7}2\mu(\tr S) 
=(n-2)\mu^2+\tr (S^2)+(n-2)\|v\|^2.\end{equation}
Let us now write $v=\|v\|e$ for some unit vector in $\mh$. By \eqref{a2}, we can write $S=2\mu e\otimes e+T$, where $T$ is a symmetric endomorphism of $e^\perp$. Moreover, we can write $T=T_0+\frac{\tr T}{n-2}\id_{e^\perp}$ with $T_0$ trace-free, and we have $\tr S=2\mu+\tr T$ and $\tr(S^2)=4\mu^2+\tr T^2=4\mu^2+\tr T_0^2+\frac{(\tr T)^2}{n-2}$. Introducing in \eqref{a7} we obtain
$$2\mu(2\mu+\tr T)=(n-2)\mu^2+4\mu^2+\tr T_0^2+\frac{(\tr T)^2}{n-2}+(n-2)\|v\|^2,$$
which can be written as 
$$(\tr T+(n-2)\mu)^2+(n-2)\tr T_0^2+(n-2)^2\|v\|^2=0.$$
This is a contradiction, thus showing that $v=0$.

Therefore, if $\nabla$ is Weyl-Einstein, then $\th=\mu b$ for some $\mu \in \R$. Using \eqref{eq:Ricgen} we get like before that $\nabla$ is a Weyl-Einstein structure if and only if \eqref{eq:a1} holds with $v=0$, that is,
\begin{multline}\label{eq:a7}
-\tr (S^2)b\otimes b+[A,S]-\tr S \cdot S
=\frac1{n}\left(-\tr(S^2)-(\tr S)^2 +(n-2)(\mu\tr S+\mu^2)\right)\Id\\-(n-2)\left[\mu S +\mu^2 b\otimes b\right].
\end{multline}
Projecting onto $\R b\otimes \R b$ and $\mh\otimes\mh$, this equation is equivalent to  the system
\begin{eqnarray}
\ -\tr (S^2)&
=&\frac1{n}\left(-\tr(S^2)-(\tr S)^2 +(n-2)(\mu\tr S+\mu^2)\right)-\mu^2(n-2)\label{eq:a8}\\
\,[A,S]&=&\frac1{n}(-\tr(S^2)-(\tr S)^2+(n-2)(\mu\tr S+\mu^2))\Id_{\mh}+(\tr S-(n-2)\mu) S.\label{eq:a9}
\end{eqnarray}
%(n-1)\tr(S^2)&=&(n-1)(n-2)\mu^2+(\tr S)^2-(n-2)\mu\tr S
Using \eqref{eq:a8} in \eqref{eq:a9} and rewriting \eqref{eq:a8}, we get that $\nabla$ is a  Weyl-Einstein structure if and only if the following system of equations holds
 \begin{eqnarray}
\tr(S^2)-\mu^2(n-2)&=&(\tr S-(n-2)\mu)\tr S/(n-1)\label{eq:a10}\\
\,[A,S]&=&(-\tr (S^2)+\mu^2(n-2))\Id_{\mh}+(\tr S-(n-2)\mu)S.
\label{eq:a11}
\end{eqnarray} 

Assume first that $S=k\Id_\mh$ for some $k\in\R$. In this case, \eqref{eq:a10} implies \eqref{eq:a11}, and $S$ satisfies \eqref{eq:a10} if and only if $\mu=0$ or $\mu=k$. This gives  \eqref{it:id} in the statement.

Now suppose $S$ is not a multiple of the identity and write $S=\frac{\tr S}{n-1}\Id_{\mh}+S_0$, with $S_0\neq 0$. Using \eqref{eq:a10} in \eqref{eq:a11}, one gets that \eqref{eq:a11} holds if and only if
\begin{eqnarray}
\,[A,S_0]&=&(\tr S-(n-2)\mu)S_0.
\label{eq:a13}
\end{eqnarray} 
Since $S_0\neq 0$ and $[A,S_0]=[A,S]$, this is equivalent to $[A,S]=0$ and $\tr S=(n-2)\mu$. Therefore, the system
\eqref{eq:a10}--\eqref{eq:a11} becomes
 \begin{eqnarray*}
 \tr S&=&\mu(n-2)\\
 (\tr S)^2&=&\tr(S^2)(n-2)\\
\,[A,S]&=&0.
\end{eqnarray*}
We thus get the case \eqref{it:S0} of the statement. 

Finally, notice that $S=k\Id_\mh$ satisfies $(n-2)\tr(S^2)=(\tr S)^2$ if and only if $k=0$; this implies that the cases \eqref{it:id} and \eqref{it:S0} are indeed exclusive.
\end{proof}

This theorem provides a construction method of almost abelian conformal Lie groups $(G,c)$ carrying left-invariant Weyl-Einstein structures, as we show next.

Let $(\mh,g_0)$ be an inner product vector space of dimension $n-1$ and let $A$ and $S$ be, respectively,  a skew-symmetric and  a symmetric  endomorphism of $(\mh,g_0)$ satisfying one of the two conditions in Theorem \ref{pro:aa}. 

Consider the Lie algebra $\mg$ which is the semidirect product of $\R$ and $\mh$ by $A+S$; namely, $\mg=\R b\ltimes \mh$ where $\ad_b=A+S$. Set $g$ to be the inner product on $\mg$ extending $g_0$ on $\mh$ and satisfying $g(b,\mh)=0$ and $g(b,b)=1$.

If $G$ is the simply connected Lie group corresponding to $\mg$ and $c$ is the conformal class of the left-invariant metric $g$ on $G$, then $(G,c)$ admits a Weyl-Einstein structure. Indeed, the Weyl structure $\nabla$ defined by the left-invariant Lee form $\th= \mu b$, where $\mu$ is given in the theorem depending on each case,  corresponding to the metric $g$ via \eqref{eq:weylrel} is Weyl-Einstein.

If $S$ is chosen to be a multiple of the identity the Riemannian Lie group $(G,g)$ is Einstein. The next example shows how to construct non-Einstein examples.

\begin{ex}\label{ex:S0}
Consider an inner product vector space $(\mh,g_0)$ of dimension $n-1$ and let $S_0$ be a non-trivial trace free symmetric endomorphism. Let $a\neq 0$ be a real number such that $ 0<\tr(S_0^2)=a^2$ and define $S:=a\sqrt{\frac{n-2}{n-1}}\Id+S_0$, which is also symmetric. One can easily check that $(n-2)\tr(S^2)=(\tr S)^2$.

Let $(\mg,g)$ be the metric Lie algebra built as a semidirect product of $\R$ and $\mh$ by $S$ as above. By construction, $(\mg,g)$ satisfies the second condition in Theorem \ref{pro:aa}, so the simply connected Lie group $G$ corresponding to $\mg$, together with  the conformal class $c:=[g]$, admits a Weyl-Einstein structure whose Lee form with respect to the metric $g$ is  $\th= a\sqrt{\frac{n-1}{n-2}} b$. Remarkably, $(G,g)$ is not Einstein, due to Lemma \ref{re:Sid}.
\end{ex}

\begin{re} One can easily check that if a nilpotent endomorphism $\ad_b=A+S$ satisfies one of the two conditions in Theorem \ref{pro:aa}, then $A=S=0$. Indeed, if  $A+S$ is nilpotent, then $0=\tr(A+S)=\tr S$. So if $A$ and $S$ satisfy either of the conditions in Theorem \ref{pro:aa}, then $S$ must vanish. Therefore,  $\ad_b=A$ is nilpotent and skew-symmetric, so it has to vanish as well.

This fact corroborates Proposition \ref{pro:nilWE}, stating that the only nilpotent conformal Lie groups admitting left-invariant Weyl-Einstein structures are the abelian ones. 
\end{re}

\begin{re}
Assume that  $\mg$ is an almost abelian Lie algebra which is not nilpotent. Then the codimension 1 abelian ideal $\mh$ is unique \cite[Theorem 2.4]{Mil76}. Given a metric $g$ on $\mg$, the endomorphisms $A$ and $S$ are thus determined up to sign, because the unit vector $b$ is unique up to sign.  Of course, the conditions in Theorem \ref{pro:aa} do not depend on the choice of the sign.
\end{re}

The Lee forms $\theta$ occurring in Theorem \ref{pro:aa} are of the form $\th=\mu b$, and thus $\mg'\subset\ker(\th)$. Therefore, $\th$ is a closed left-invariant 1-form on $G$, so since $G$ is simply connected, $\th={\rm d} f$ for some differentiable function $f$ on $G$. Hence,  the Weyl-Einstein structure $\nabla$ is the Levi-Civita connection of the  Riemannian metric  $h:=e^{2f}g\in c$, i.e. $\nabla=\nabla^h$.

 This observation leads to the following consequence of Theorem \ref{pro:aa}.

\begin{co}\label{co:WEisc}
Let $(G,c)$ be a simply connected almost abelian conformal Lie group. Every left-invariant Weyl-Einstein structure on $(G,c)$ is the Levi-Civita connection of an Einstein Riemannian metric in $c$.
\end{co}

Notice that the converse may not hold in general since an Einstein metric $h=e^{2f} g$ with $g$ left-invariant may not satisfy that $\th:={\rm d} f$ is left-invariant, and thus the Weyl-Einstein structure would fail to be left-invariant.

We now compute the curvature of the Einstein metrics that arise from Weyl-Einstein structures on simply connected almost abelian conformal Lie groups.

\begin{pr}\label{pro:rff}
Let $(G,c)$ be a simply connected almost abelian conformal Lie group admitting a left-invariant Weyl-Einstein structure $\nabla$ and let $g$ be a left-invariant metric in $c$. Denote by $\th$ the left-invariant Lee form of $\nabla$ with respect to $g$, and assume that $\th\neq 0$. Then the Riemannian metric $h:=e^{2f}g$, where $f$ verifies $\th={\rm d} f$, is Ricci-flat. 

Moreover, denoting $A$ and $S$ the endomorphisms associated to $(\mg,g)$, one has that $h$ is flat if and only if either $S=k\Id_\mh$ for some $k\in\R$ or there is a codimension $1$ subspace $U\subset \mh$ such that $S|_U=\alpha\Id_U$ for some $\alpha\in\R$, and $S|_{U^\bot}=0$.
\end{pr}

\begin{proof}Let  $\nabla$ denote a left-invariant Weyl-Einstein structure on $(G,c)$, let $g\in c$ be left-invariant, so that $(\mg=\R b\ltimes \mh,g)$ is the corresponding metric Lie algebra. From Theorem \ref{pro:aa}, we know that  that the Lee form $\th\in \mg^*$ has the form $\th=\mu b$, where $\mu\in\R$ depends on the endomorphisms $A$ and $S$ of $(\mg,g)$. In any case, ${\rm d}\th=0$ and, $G$ being simply connected, there is a differentiable function $f$ such that $\th={\rm d} f$ and $h:=e^{2f} g$ is an Einstein metric. Then $\nabla=\nabla^h\neq \nabla^g$, because we assume $\th\neq 0$.

We compute $\ric{\nabla^h}$ by using the formula of conformal change of metrics  \cite[\S1.159]{Besse2008}, which reads
\begin{equation}\label{eq:rich}
 \ric{\nabla^h} =  \ric{\cn{g}} - (n-2)\bigl(\cn{g}\th - \th \otimes \th\bigr) + \bigl(\de^g\th - (n-2)\|\th\|_g^2\bigr)g.
\end{equation}
 Note that this coincides with the expression of $\ric{\nabla}$ given in \eqref{eq:ricrel}.
Since $\th$ is closed, we have 
$\nabla^g\th=(\nabla^g\th)^{\rm sym}$, so using  \eqref{eq:nabth} and the fact that $\th=\mu b$ we further get
\[
\cn{g}\th - \th \otimes \th=-\ad_\th^{\rm{sym}}-\mu^2b \otimes b=-\mu S-\mu^2b \otimes b, \quad 
\de^g\th=\tr \ad_\th=\mu \tr S.
\]
These equalities together with \eqref{eq:ricaa} and \eqref{eq:rich} imply
\begin{equation}
 \ric{\nabla^h}= -\tr(S^2)b\otimes b+[A,S]-\tr S\cdot S - (n-2)\bigl(-\mu S-\mu^2b \otimes b\bigr) + \bigl(\mu\tr S -\mu^2 (n-2)\bigr)\Id.
\end{equation}In particular, $\ric{\nabla^h}$ preserves both $\R b$ and $\mh$.
The part of $\ric{\nabla^h}$ on $\R b\otimes \R b$  is
\begin{equation}\label{eq:bb1}
(-\tr(S^2)+ (n-2)\mu^2+\mu \bigl(\tr S -\mu(n-2)\bigr))b\otimes b
\end{equation}
and the part on $\mh\otimes \mh$ can be written as
\begin{equation}\label{eq:AS} [A,S]+  \bigl(\tr S -\mu (n-2)\bigr)(\mu\Id_\mh-S).
\end{equation}
Recall that $A$, $S$ and $\mu$ verify one of the two conditions in Theorem \ref{pro:aa}. 

On the one hand, if $S=k\Id$ for some $k\in \R$, then one can have $\mu=0$ or $\mu=k$. However, since we assume $\mu\neq 0$, the only possibility is $\mu=k$ and $S\neq 0$. It is easy to check that in this case \eqref{eq:bb1} and \eqref{eq:AS} vanish and thus $\ric{\nabla^h}=0$.

On the other hand, if $S\neq 0$, $(\tr S)^2=\tr(S^2)(n-2)$ and $[A,S]=0$, then $\mu=\tr S/(n-2)$ and thus one has that \eqref{eq:bb1} and \eqref{eq:AS} vanish again. Therefore $h$ is Ricci-flat in both cases.

In order to prove the second part of the statement, we compute the curvature tensor $R^h$ by using the formulas relating the curvature of two conformal metrics \cite[\S1.159]{Besse2008}. From these formulas, it is easy to see that $R^h=0$ if and only if the following equation holds
\begin{equation}\label{eq:KN}
R^g=g\owedge(\nabla\th-\th\otimes\th+\frac12|\th|^2g),\end{equation}
where $R^g$ is the curvature of $(G,g)$ and $\owedge$ denotes the Kulkarni-Nomizu product (cf. \cite[\S1.110]{Besse2008}).

The curvature $R^g$ on left-invariant vector fields is given in \eqref{eq:Raa}, so we shall compute the right hand side of \eqref{eq:KN}. Using that $\th=\mu b$ is closed, 
$\nabla^g\th=(\nabla^g\th)^{\rm sym}$ and  from \eqref{eq:nabth}  we get
\begin{equation}\label{eq:Tkk}
T:=g\owedge(\nabla\th-\th\otimes\th+\frac12|\th|^2g)=g\owedge(-\mu S-\mu^2 b\otimes b+\frac12 \mu^2 g)
\end{equation}
We prove first that $R^g=T$ in the two cases in the statement, that is, when $S=k\Id_\mh$ for some $k\in\R$ and also if there is a codimension 1 subspace $U\subset \mh$ such that $S|_U=\alpha\Id_U$ for some $\alpha\in\R$, and $S|_{U^\bot}=0$.

Suppose first that  $S= k \Id$ for some $k\in \R$. Since $\th=\mu b$ defines a Weyl-Einstein structure, we know from Theorem \ref{pro:aa} that $\mu=k$, which we assume to be non-zero.  Then one can easily check that $T=-\frac12 k^2 g\owedge g$ and thus, by using \eqref{eq:Raa}, one gets that $R^g=T$. Therefore, \eqref{eq:KN} holds and thus $R^h$ is flat.

Now assume that there is a subspace $U\subset \mh$ of codimension 1 such that $S|_U=\alpha\Id_U$ for some $0\neq \alpha\in\R$ and $S|_{U^\bot}=0$; note that $(\tr S)^2=\tr(S^2)(n-2)$. According to Theorem \ref{pro:aa}, $\mu=\tr S/(n-2)=\alpha$ since $\th$ is the Lee form of a Weyl-Einstein structure. With these elements, one can explicitly compute the curvature $R^g$ using \eqref{eq:Raa} and the tensor $T$ in \eqref{eq:KN}, to show that $R^g=T$.  Hence $R^h=0$ in this case as well.

Conversely, suppose $\th=\mu b$ with $\mu\neq 0$ defines a left-invariant Weyl-Einstein structure with $R^h=0$.  
We will show that if $S\neq k\Id$ for any $k\in\R$, then there is a subspace $U$ satisfying the conditions in the statement. 

Assume that $S\neq k\Id$ for any $k\in\R$, then $S,A$ satisfy \eqref{it:S0} in Theorem \ref{pro:aa}. In particular, $[A,S]=0$. For every $v\in\mh$, \eqref{eq:Tkk}  gives
\[
T_b(v,v):=T(b,v,b,v)=-\mu g(Sv,v),
\]and using  \eqref{eq:Raa} and $[A,S]=0$, we have \[
R_b(v,v):=g(R^g(b,v)b,v)=-g(S^2 v,v).\]
Since $R^h=0$, $T=R^g$ by \eqref{eq:KN} and thus $T_b(v,v)=R_b(v,v)$ for every $v\in\mh$. In particular, if $v$ is an eigenvector of $S$ associated to the eigenvalue $\lambda$ this equality gives
\[
-\mu \lambda |v|^2=T_b(v,v)=R_b(v,v)=-\lambda^2 |v|^2.
\]
This implies that the only non-zero eigenvalue of $S$ is $\mu$. Since $S\neq 0$,  $(\tr S)^2=\tr(S^2)(n-2)$ and $\tr S=\mu(n-2)$, one must have $\dim\ker S=1$, and $S|_{(\ker S)^\perp}=\mu\Id_{(\ker S)^\perp}$ completing the proof.
\end{proof}

\begin{ex} 
Consider an inner product vector space $(\mh,g_0)$ of dimension $n-1\geq 3$ and let $S_0$ be a non-trivial trace free symmetric endomorphism such that every eigenspace of $S_0$ is  of dimension at most $n-3$.

Let $\nabla$ be the Weyl-Einstein structure on the simply connected Riemannian solvable Lie group corresponding to $(\mg=\R b\ltimes \mh,g)$ built from $S_0$ as in Example \ref{ex:S0} and let $\th$ be the left-invariant Lee form corresponding to $\nabla$ by $g$. 

By Proposition \ref{pro:rff}, the Weyl structure $\nabla$ is not flat. Thus, the Riemannian metric $h=e^{2f} g$, where $f$ satisfies ${\rm d}f=\th$, is Ricci-flat but not flat. Notice, that $h$ is not left-invariant.
\end{ex}

\section{3-dimensional solvable Riemannian Lie groups}\label{ch:3d}

In this section we apply the above results in order to 
classify simply connected solvable Riemannian Lie groups $(G,g)$ of dimension 3 carrying left-invariant Weyl-Einstein structures. This is possible because of the following simple observation:
%, by performing a simple inspection in the classification of solvable Lie algebras in low dimension \cite[\S I.4]{JA}, one can easily check that every simply connected solvable Lie group of dimension 3 is almost abelian.   

\begin{lem} Every solvable Lie algebra of dimension $3$ is almost abelian.
\end{lem}

\begin{proof} If $\mg$ is a 3-dimensional solvable Lie algebra, its derived algebra $\mg'$ is a nilpotent ideal of dimension at most 2, so it is abelian. 

If $\dim(\mg')=2$, $\mg'$ is then a codimension 1 abelian ideal. If $\dim(\mg')=0$, $\mg$ is abelian. Finally, if $\dim(\mg')=1$, let $\xi$ be a generator of $\mg'$ and let $\mm$ be the kernel of the linear map $\mg\to\mg'$ given by $x\mapsto [x,\xi]$. Since $\dim(\mm)\ge 2$, there exists a vector $\zeta\in \mm\setminus \mg'$. Then $\xi$ and $\zeta$ span a codimension 1 abelian ideal of $\mg$.
\end{proof}

Left-invariant Riemannian metrics on 3-dimensional simply connected Lie groups were classified, up to automorphisms, by Ha and Lee \cite{HaLe}. Recall that, for simply connected Lie groups, the classes of left-invariant metrics up to Lie group automorphisms are in one-to-one correspondence with classes of metric Lie algebras $(\mg,g)$ up to Lie algebra automorphisms.

In what follows we introduce some notation and review the results in \cite{HaLe} which are of interest for us, namely those corresponding to solvable Lie algebras which are either abelian, or solvable and non-nilpotent (see Proposition \ref{pro:nilWE}). 

Let $\mg$ be a solvable Lie algebra of dimension 3. 
If $\mg$ is abelian, then every metric $g$ on $\mg$ is equivalent, up to automorphisms, to the standard metric $g_\bullet$. Now, if $\mg$ is neither nilpotent nor abelian, then $\mg$ is isomorphic to  a Lie algebra having a basis $\mathcal B=\{x,y,z\}$ whose Lie brackets satisfy one of the following:
\begin{itemize}
\item $\mg=\mathfrak{Sol}$: $[x,y]=0$, $[z,x]=x$, $[z,y]=-y$ ,
\item $\mg=\mathfrak{so}(2)\ltimes \R^2$: $[x,y]=0$, $[z,x]=-y$, $[z,y]=x,$ 
\item $\mg=\R\ltimes_{\Id} \R^2$: $[x,y]=0$, $[z,x]=x$, $[z,y]=y$,
\item $\mg=\mg_t$: $[x,y]=0$, $[z,x]=y$, $[z,y]=-tx+2y$, for some $t\in\R$.
\end{itemize}

We shall denote $\{x^*,y^*,z^*\}$ the dual basis of $\mathcal B$ and the symmetric product of two covectors $\xi,\zeta\in\mg^*$ by $\xi\odot \zeta:=\frac12 (\xi\otimes \zeta+\zeta\otimes \xi)$. 
Consider the following symmetric bilinear forms on $\mg$:
\[\begin{array}{lcl}
g_\nu&=& x^*\odot x^*+y^*\odot y^*+\nu \,z^*\odot z^*\\
g_{\mu,\nu}&=&  x^*\odot x^*+\mu\, y^*\odot y^*+\nu \,z^*\odot z^*\\
h_{\mu,\nu}&=&   x^*\odot x^*+2x^*\odot y^* +\mu \, y^*\odot y^*+\nu \,z^*\odot z^*\\
m_{\nu}&=&   x^*\odot x^*+x^*\odot y^* + y^*\odot y^*+\nu \,z^*\odot z^*
\end{array}
\] 
where $\mu,\nu$ are real parameters such that the above are indeed positive definite.

By \cite{HaLe}, every inner product $g$ on $\mg$ is, up to an automorphism, one of the following:
\begin{itemize}
\item if $\mg=\R\ltimes_{\Id} \R^2$,  $g=g_{\nu}$ with $ 0<\nu$.
\item if $\mg=\mathfrak{so}(2)\ltimes \R^2$,  $g=g_{\mu,\nu}$ with $ 0<\nu$ and $0<\mu\leq 1$.
\item if $\mg=\mg_t$ where $t>1$, $g=h_{\mu,\nu}$ with $0<\nu$ and $0<\mu 	\leq t  $.
\item if $\mg=\mg_0$,  $g=g_{\mu,\nu}$ or $g=m_{\nu}$ with $0<\mu ,\nu$.
\end{itemize}
One should notice that in \cite{HaLe}, the metrics on $\mathfrak{Sol}$ and on $\mg_t$ for any value of $t\in\R$  are classified. However, we do not give the full classification for these cases since they will not appear in our results.

\begin{theorem}\label{th:cl3}
Let $(G,g)$ be a solvable and simply connected Riemannian Lie group of dimension $3$ with Lie algebra $\mg$. Then $(G,[g])$ admits a left-invariant Weyl-Einstein structure if and only if $(\mg,g)$ is one of the following metric Lie algebras
\begin{enumerate}
\item $(\R^3,g_\bullet)$.
\item $(\R\ltimes_{\Id}\R^2,g_{\nu})$, for any $\nu>0$.
\item $(\mathfrak{so}(2)\ltimes \R^2,g_{1,\nu})$, for any $\nu>0$.
\item $(\mg_t,h_{t,\nu})$ for $t>1$ and for any $\nu>0$.
\item $(\mg_0,m_\nu)$ for  any $\nu>0$.
\end{enumerate}
\end{theorem}

A possible proof of this theorem would be to go through the list of metric solvable Lie algebras of dimension 3 given by Ha and Lee  \cite{HaLe}, to compute the endomorphisms $A$ and $S$ for each case, and then to apply Theorem \ref{pro:aa}. However, we proceed by a constructive approach which relies on the following preliminary result.

\begin{lem}\label{lem:buv} Let $(G,g)$ be a solvable and simply connected Riemannian Lie group of dimension $3$ with Lie algebra $\mg$. Then $(G,[g])$ admits a left-invariant Weyl-Einstein structure if and only if $\mg$ admits an orthonormal basis $\{b,u,v\}$ whose Lie brackets satisfy one of the following relations
\begin{enumerate}
\item \label{it:adb} there are $k,l\in\R$ such that $[b,u]=ku-lv$, $[b,v]=lu-kv$ and $[u,v]=0$;
\item \label{it:dir} there is $0\neq \alpha\in\R$ such that  $[b,u]=\alpha u$, $[b,v]=[u,v]=0$.
\end{enumerate}
\end{lem}
\begin{proof} Since  $\mg$ is solvable of dimension 3,  it is almost abelian. So let $\mh$ be an abelian ideal of dimension 2 and let $b$ be a unit norm vector spanning $\mh^\bot$. Denote $S$ and $A$ the symmetric and skew-symmetric parts of $\ad_b|_{\mh}$, respectively.

According to Theorem \ref{pro:aa} and Corollary \ref{co:WEisc}, $G$ admits an Einstein metric conformal to $g$ if and only if either $S=k \Id_\mh$ for some $k\in \R$, or $(\tr S)^2=\tr(S^2)$ and $[A,S]=0$, with $S\neq 0$.

Assume  $S=k \Id_\mh$ for some $k\in \R$ and let $\{u,v\}$ be an orthonormal basis of $\mh$. It is easy to check that the matrix of $\ad_b$ in this basis has the form  
\begin{equation}\label{eq:adbuv}
[\ad_b]_{\{u,v\}}=\left(\begin{matrix}k&l\\-l&k\end{matrix}\right),
\end{equation}
for some $l\in\R$, thus giving \eqref{it:adb}.

Assume now that $(\tr S)^2=\tr(S^2)$ and $[A,S]=0$ with $S\neq 0$. Since $S$ is not a multiple of the identity and $\dim \mh=2$, $[A,S]=0$ implies $A=0$. Let $\{u,v\}$ be an orthonormal basis of $\mh$ of eigenvectors of $S$ and let $\alpha\neq \beta$ be the respective eigenvalues. The condition $(\tr S)^2=\tr(S^2)$ implies $\alpha\beta=0$; without loss of generality, we may assume $\beta=0$. Hence, the Lie bracket satisfies $[b,u]=\alpha u$, $[b,v]=[u,v]=0$. 
\end{proof}

\begin{proof}[Proof of Theorem \ref{th:cl3}]
Assume that $(G,[g])$ admits a left-invariant Weyl-Einstein structure and let $\{b,u,v\}$ be the orthonormal basis of $(\mg,g)$ satisfying one of the conditions in Lemma \ref{lem:buv}. We shall make a change of this basis to show that $(\mg,g)$ is indeed one of the metric Lie algebras listed in the statement.

Suppose that we are in case \eqref{it:adb} of Lemma \ref{lem:buv}  and there are $k,l\in\R$ such that \eqref{eq:adbuv} holds. 
If $k=l=0$, then $\mg$ is abelian and thus $g$ is isometrically isomorphic to $(\mg,g_\bullet)$. In the case $l=0$, and $k\neq 0$, consider the basis 
\[
z:=\frac1{k} b,\quad x:=u,\quad y:=v.
\]
It is straightforward to check that $\{x,y,z\}$ verifies the Lie bracket relations of $\R\ltimes_{\Id}\R^2$ and the metric $g$ in this basis takes the form $g_\nu$ with $\nu=k^{-2}$. To the contrary, when $k=0$ and $l\neq 0$, the basis 
\[
z:=\frac1{l} b,\quad x:=u,\quad y:=v.
\] verifies the Lie bracket relations of $\mathfrak{so}(2)\ltimes \R^2$ and the metric $g$ in this basis takes the form $g_{1,\frac1{l^2}}$.

Now assume that $k,l\neq 0$, and consider the linearly independent vectors
\[
z:=\frac1{k} b,\quad x:=u,\quad y:=u-\frac{l}{k}v.
\]
Then, 
\[
[\ad_z]_{\{x,y\}}=\left(\begin{matrix}0&-(1+\frac{l^2}{k^2})\\1&2\end{matrix}\right),\qquad [g]_{\{x,y,z\}}=\left(\begin{matrix}1&1&0\\1&1+\frac{l^2}{k^2}&0\\0&0&\frac1{k^2}\end{matrix}\right),
\]
so we get that $\mg=\mg_t$ with $t:=1+\frac{l^2}{k^2}>1$ and the metric $g=h_{t,\frac1{k^2}}$.

Assume now that $\mg$ admits a $g$-orthonormal basis $\{b,u,v\}$ satisfying \eqref{it:dir} of Lemma \ref{lem:buv}, for some $\alpha\neq 0$. Consider the basis
\[
z:=\frac{2}{\alpha} b,\quad x:=-\frac12 u -\frac{\sqrt3}2 v,\quad y:=-u,
\] for which one has
\[
[\ad_z]_{\{x,y\}}=\left(\begin{matrix}0&0\\1&2\end{matrix}\right),\qquad [g]_{\{x,y,z\}}=\left(\begin{matrix}1&\frac12&0\\\frac12&1&0\\0&0&\frac4{\alpha^2}\end{matrix}\right).
\]
It is easy to check that $\mg=\mg_0$ with metric $g=g_{\frac4{\alpha^2}}$, so we get the only remaining case.

Conversely, if $(\mg,g)$ is one of the metric Lie algebras in the statement, reversing the changes of bases above and using Lemma \ref{lem:L3}, one can easily check that $(G,[g])$ admits a left-invariant Weyl-Einstein structure.
\end{proof}

\bibliographystyle{plain}
\bibliography{biblio}

\end{document}